\documentclass{article}
\usepackage[preprint,nonatbib]{neurips_2019}
\usepackage[utf8]{inputenc}

\usepackage{nicefrac}

\usepackage{amsmath}
\usepackage{amsthm}
\usepackage{amssymb}
\usepackage{amsfonts}

\usepackage{booktabs} 
\usepackage{multirow}

\usepackage{url}

\usepackage{listings}

\usepackage{graphicx}

\newtheorem{proposition}{Proposition}
\newtheorem{lemma}{Lemma}
\newtheorem{claim}{Claim}
\newtheorem{theorem}{Theorem}
\newtheorem{remark}{Remark}
\newtheorem{corollary}{Corollary}

\usepackage{hyperref}
\usepackage{cleveref}

\crefname{proposition}{Proposition}{Propositions}
\crefname{claim}{Claim}{Claims}

\title{Revisiting Concentration of Missing Mass}
\author{Maciej Skorski}

\begin{document}

\maketitle

\begin{abstract}
We revisit the problem of \emph{missing mass concentration}, developing a new method of estimating concentration of heterogenic sums, in spirit of celebrated Rosenthal's inequality.
As a result we slightly improve the state-of-art bounds due to Ben-Hamou at al., and simplify the proofs.
\end{abstract}

\section{Introduction}

\subsection{Missing Mass Problem}

Imagine drawing $n$ independent samples $S_n=(Y_1,\ldots,Y_n)$ from a distribution $Y$ over $1,2,\ldots,m$. 
The \emph{missing mass}
is defined as the total weight of \emph{not sampled values}
\begin{align}\label{eq:miss_mass}
M= \sum_i \Pr[Y=i] \cdot \mathbb{I}(i\not \in S_n).
\end{align}
Estimating the missing mass is important in ecology (existence of species not observed in a specific sample) and natural language modeling (encountring out-of-vocabulary words). Practical estimates are obtained by the Good-Turing estimator~\cite{acharya2018improved}.
In this paper we are interested in obtaining \emph{good confidence intervals} for $M$. 

The first obstacle is that $M$ is a sum of dependent components. Indeed, occurrences of two different elements in a sample are \emph{negatively correlated}. Fortunately, this dependence actually helps to shrink confidence intervals:
utilizing \emph{negative dependence} theory~\cite{dubhashi1996balls,joag1983negative} one obtains the stochastic domination~\cite{ortiz2003concentration}
\begin{align}\label{eq:domination}
M \leqslant^{MGF} M'=\sum_i  X_i,\quad   X_i \overset{d}{=} \Pr[Y=i]\cdot \mathbb{I}_{i\not\in S_n }\text{ and } X_i \text{ are independent}.
\end{align}
so the task essentially reduces to the well-investigated problem of studying sums of independent random variables. 
Unfortunately popular concentration inequalities due to Chernoff, Hoefdding, Bernstein etc. do not give good results 
because of \emph{heterogeneous summands}~\cite{ortiz2003concentration}. 

In particular, improving the original weaker bounds~\cite{mcallester2000convergence} to gaussian-like tails
\begin{align}\label{eq:gaussian_tails}
\Pr[ | M -\mathbf{E}M| > \epsilon ] \leqslant \mathrm{e}^{-\Omega(n\epsilon^2)}.
\end{align}
due to ~\cite{ortiz2003concentration},
with constants improved in subsequent works~\cite{mcallester2003concentration,ortiz2003concentration,berend2013concentration,DBLP:conf/isit/ChandraT19},  took considerable effort relying on deep facts (sharp logarithmic Sobolov inequalities) for subgaussian norms~\cite{kearns2013large,buldygin2013sub}.
The state-of-art bounds were found recently in~\cite{ben2017concentration} and depends on some distribution-dependent constants
\begin{align}\label{eq:gaussian_tails}
\Pr[ | M -\mathbf{E}M| > \epsilon ] \leqslant \mathrm{e}^{-\Omega\left(\frac{\epsilon^2}{v^2(Y) + b(Y)\epsilon }\right)}
\end{align}
Under some conditions on the distribution $Y$ this is sharp up to absolute constants in the exponent.

In this paper we develop a very elementary Rosenthal-type concentration bound for sums of heterogenic terms (of independent interests and with more applications), and use it to analyze the missing mass.
In arguably simpler manner we obtain the state-of-art bounds, and demonstrate better results for some theoretical settings and quantitatively in numerical evaluation.


\subsection{Our Contribution}


\subsubsection{Mistake in Exponential Bounds~\cite{khanloo2015novel}}

We note that there is a result~\cite{khanloo2015novel} \emph{much superior} tail bounds of $\mathrm{e}^{-\tilde{\Omega}(n\epsilon)}$. It was ignored in most recent works~\cite{DBLP:conf/isit/ChandraT19,ben2017concentration} and likely contains a mistake.
The central idea in~\cite{khanloo2015novel} is to re-group the terms in \Cref{eq:miss_mass} by considering another distribution $Y'$ which splits or combines weights of $Y$ in such a way that the stochastic domination still holds. The aim is to make the weights more homogeneous so that standard concentration inequalities can be successfully applied. This is different from other works
~\cite{mcallester2003concentration,ortiz2003concentration,berend2013concentration,DBLP:conf/isit/ChandraT19,ben2017concentration} which work with $Y$ as it is given. 

However the technique is critically based on Lemma 9 which claims that "absorption" preserves negative dependency. 
More precisely consider 
$w_i = \sum_j [Y_j = i]$, the number of occurrences of $i$; then $w_i$ are negatively dependent; the absorption is understood as redistributing the mass from one bin evenly among others, e.g. $w'_i = w_i + \frac{w_m}{m-1}$ for $i=1,\ldots,m-1$. The lemma claims that 
$w'_i$ are still negatively dependent by invoking the result for the IID case with $m-1$ symbol, but the modified sample doesn't meet the IID assumptions.

\subsubsection{Technical Result: Bounds for Heterogenic Sums}

Although previous works~\cite{mcallester2003concentration,ortiz2003concentration,berend2013concentration,DBLP:conf/isit/ChandraT19}
emphasize the difficulty in applying "standard bounds", none of them gave a try to well-known bounds developed for \emph{independent heterogenic sums}. We recall the celebrated result due to Rosenthal, with optimal constants found in~\cite{ibragimov1998exact} and ~\cite{pinelis1994optimum}.
\begin{proposition}[Rosenthal's Inequality]\label{lemma:Rosenthal}
Let $X_i$ be zero-centered random variables and $k\geqslant 2$. Then
\begin{align}
\left(\mathbf{E} \left|\sum_i X_i\right|^k\right)^{1/k} \leqslant C \max\left( \sum_i \left(\mathbf{E} X_i^2\right)^{1/2},  \left(\sum_i \mathbf{E} |X_i|^k\right)^{1/k}\right)
\end{align}
where $C =O( k/\log k)$. One can also take asymmetric constants
\begin{align}
\left(\mathbf{E} \left|\sum_i X_i\right|^k\right)^{1/k} \leqslant c \sum_i \left(\mathbf{E} X_i^2\right)^{1/2} + C \left(\sum_i \mathbf{E} |X_i|^k\right)^{1/k}
\end{align}
where $c= O(k^{1/2})$ and $C = O(k)$. 
\end{proposition}
Already this inequality gives gaussian tails for the missing mass. 
To further improve we develop the following inequality, which needs to control same terms as in Rosenthal's inequality. The proof is basically one application of the AM-GM inequality!
\begin{theorem}[Bounds for heterogenic independent sums]\label{main:lemma}
Let $X_i$ for $i=1,2,\ldots,m$ be zero-centered independent random variables. Then for any real $t$ it holds that
$$
\mathbf{E} \exp\left( t\sum_{i=1}^{m} X_i \right) \leqslant  \left(1+\frac{1}{m}\sum_{k\geqslant 2}\frac{t^k}{k!} \cdot \sum_{i=1}^{m} \mathbf{E}X_i^k \right)^m.
$$
Moreover, the equality is achieved when $X_i$ are identitically distributed.
\end{theorem}
\begin{remark}[Further improvements]
The AM-GM inequality can be refined for non-identical variables, for example due to \emph{self-improving} properties~\cite{aldaz2008selfimprovemvent}.
This can be used to improve the confidence interval obtained from \Cref{main:lemma} in numerical computations.
\end{remark}

\subsubsection{Application: Confidence Bounds for Missing Mass}

As a corollary from \Cref{lemma:Rosenthal} we obtain
\begin{corollary}[Concentration of Missing Mass]\label{cor:rosenthal}
For $\epsilon = O(n^{-1/3})$ we have that
\begin{align}
\Pr[|M-\mathbf{E}M| > \epsilon] \leqslant \exp\left(-\Omega(n\epsilon^2)\right).
\end{align}
\end{corollary}
Note that this result already has practical applications because the standard deviation of $M$ can be as big as $O(n^{-1/2})$ which is much smaller than the covered range $\epsilon=O(n^{-1/3})$.

From our key technical result, \Cref{main:lemma}, we easily derive missing mass concentration, improving upon best known bounds~\cite{ben2017concentration}. Our derivation is arguably simpler: we just compute the right-hand side in \Cref{main:lemma}. Also both tails are handled in a single bound, as opposed to previous works. 

From now on denote $p_i = \Pr[Y=i]$ and $q_i = \Pr[i\not\in S] = (1-p_i)^n$, then $X_i  \overset{d}{=} p_i\mathrm{Bern}(q_i)$.
We show gaussian-like behaviors for tails with \emph{variance proxy} which is a weighted combination of variances of terms appearing in \Cref{eq:domination}.
\begin{theorem}[Concentration of Missing Mass]\label{cor:main}
Let $\phi = x^{-2}(\exp(x)-1-x)$, fix a real number $t$ and define the weights
\begin{align}
\theta_i = q_i \phi(t p_i q_i) + (1-q_i)\phi(t p_i (1-q_i)).
\end{align}
and the variance proxy
\begin{align}
\sigma^2 = 2\sum_i \theta_i \mathbf{Var}[X_i].
\end{align}
Then the moment generating function (MGF) of the missing mass is bounded by
\begin{align}
\mathbf{E}\exp(t(M-\mathbf{E}M)) \leqslant \left(1+\frac{t^2\sigma^2}{2m} \right)^m \leqslant \exp\left(\frac{t^2\sigma^2}{2}\right).
\end{align}
\end{theorem}

\begin{remark}[Bounds are superior wrt \Cref{eq:gaussian_tails}]
We always have $\mathbf{Var}[X_i] \leqslant p_i^2 q_i \leqslant p_i^2\mathrm{e}^{-np_i}$ and $\theta_i\leqslant \phi(|t| p_i)$ because $\phi$ is monotone, and also $\phi(|t| p_i)\leqslant \mathrm{e}^{|t| p_i}$.
Under the extra assumption $|t| \leqslant n/2$ we thus obtain $\theta_i\mathbf{Var}[X_i] \leqslant p_i^2\mathrm{e}^{-n p_i /2}$ and then $\sigma^2\leqslant \sum_i p_i^2\mathrm{e}^{-n p_i /2} = O(1/n)$ (see ~\cite{ortiz2003concentration}). 
This implies the bound of $\exp(O(t^2/n))$ when $|t|\leqslant n/2$ for the moment generating function and the 
 tail bound of $\mathrm{e}^{-\Omega(n\epsilon^2)}$ for $|\epsilon| \leqslant O(1)$ and thus for all $\epsilon$ (for values $|\epsilon| > 1$ the tail is zero because the missing mass is not bigger than 1).
\end{remark}
It is easy to bound the weights for the case of lower tails, since $\phi(x) \leqslant \frac{1}{2}$ for negative $x$.
\begin{corollary}[Simpler Bound on Lower tails]
For any $t<0$ the variance proxy is bounded by 
\begin{align}
\sigma^2\leqslant \sum_i\mathbf{Var}[X_i]
\end{align}
Therefore for $v_{-}^2 = \sum_i\mathbf{Var}[X_i]$ we have (see Prop 3.7 in~\cite{ben2017concentration})
\begin{align}
\mathbf{E}\exp(t(M-\mathbf{E}M)) \leqslant \exp(t^2v_{-}^2/2),\quad t<0.
\end{align}
By Chernoff's inequality we get the following bounds
\begin{align}
\Pr[M-\mathbf{E}M <-\epsilon]\leqslant  \exp(-\epsilon^2/2v_{-}^2),\quad \epsilon > 0.
\end{align}
\end{corollary}
\begin{remark}[Lower tails are "sharp"~\cite{ben2017concentration}]
This lower tail is considered sharp~\cite{ben2017concentration}, up to a constant in the exponent. 
This is because for the missing mass problem it is unlikely that negative association concentrate much better than the IID case, and for the IID case the sum cannot concentrate better as gaussian tails (due to Central Limit Theorem).
\end{remark}

Finding simple bounds for the weights in case of upper tails is typically more complex. We show one technique, \emph{poissonization}, which was also used to bound MGFs in~\cite{ben2017concentration}.
\begin{proposition}[Bounding weights by Poisson occupancy numbers]\label{eq:remark2}
For $i=1,2,\ldots$ consider Poisson distributions $P_i$ with expectations $n p_i$ (expected occurrences of $i$ in the sample). Let $K_r(n) = \sum_i \mathbb{I}(P_i = r)$ (counts the number of distributions that output $r$). We then have 
\begin{align}
\mathbf{E} K_r(n) = \sum_{i} \mathrm{e}^{-n p_i}\frac{(np_i)^r}{r!}
\end{align}
and the variance proxy is bounded by 
\begin{align}
\sigma^2\leqslant  2t^{-2}\sum_{r\geqslant 2}\left(\frac{t}{n}\right)^{r} \mathbf{E}K_r(n).
\end{align}
\end{proposition}
Since the right-hand side as defined above is increasing in $t$ for $t>0$ we obtain a simpler version
\begin{corollary}[Simpler Bound on Upper tails]
In particular for $v_{+}^2 = 2 n^{-2} \sum_{r\geqslant 2} \mathbf{E} K_r(n)$ one has
\begin{align}
\mathbf{E}\exp(t(M-\mathbf{E}M)) \leqslant \exp(n^2v_{+}^2/2),\quad 0<t\leqslant n.
\end{align}
\end{corollary}
Tail bounds can be obtained by Chernoff's inequality.
\begin{remark}[Comparison with ~\cite{ben2017concentration}]
We compare this with Theorem 3.9 in ~\cite{ben2017concentration}. Their result has the right-hand side of $\exp\left(\frac{t^2v_{+}^2}{2(1-t/n)}\right)$ with the same same $v_{+}$, thus our exponent is strictly better (their exponent becomes unbounded for $t$ close to $n$). \end{remark}
As a side note we point out that the upper tail in ~\cite{ben2017concentration} is missing some assumptions, for example $t > n$ is allowed but gives the MGF smaller than 1, which contradicts Jensen's inequality (centered r.vs. have MGF at least 1).

Below we point out limitations of poissonized bounds.
\begin{remark}[Limitation of Poissonization]
The heart of poissonization is the approximation 
$$\mathbf{Var}[Y_i]\approx p_i^2\mathrm{e}^{-np_i} = n^{-2}\Pr[\mathrm{Poiss}(\lambda=n p_i)=2]$$ 
which overestimates when $p_i = o(n^{-1})$. This becomes clear when we study the birthday paradox setting.
Even when $p_i  = \Omega(n^{-1})$ the approximation may still loose a constant factor which has visible impacts on confidence intervals. This becomes apparent in our numerical experiments.
\end{remark}


\subsubsection{Application: Flexible Bernstein's Inequality}

Our bound in \Cref{main:lemma} implies the \emph{Flexible Bernstein Inequality} considered in~\cite{boucheron2013concentration}. 
In vast majority of cases Bernstein's inequality is stated for bounded random variables,
but this variant is \emph{much more powerful}, as it requires only an appropriate control of moments and handles terms of different magnitude. 
This is important in the context of difficulties in missing mass estimation highlighted by previous authors, and for many other settings.
\begin{corollary}[Flexible Bernstein's Inequality]\label{cor:flexible}
Let $X_i$ be independent and such that 
\begin{align}
\sum_i \mathbf{E}(X_i^{+})^{k}\leqslant \frac{k! v^{2} b^{k-2}}{2}
\end{align}
for some positive parameters $b,v$. Let $S = \sum_i X_i$. Then we have
\begin{align}
\mathbf{E}\exp(tS) \leqslant \exp\left(\frac{v^2t^2}{1-tb} \right),\quad |t| < b^{-1}
\end{align}
which implies the tail bound
\begin{align}
\Pr[S > \epsilon] \leqslant \mathrm{e}^{-\Omega\left(\frac{\epsilon^2}{v^2 + b\epsilon}\right)}.
\end{align}
\end{corollary}
\begin{remark}[Exact constants]
The constants for the tail can be optimized as in ~\cite{boucheron2013concentration}.
\end{remark}
One can show that this inequality \emph{easily} implies \Cref{eq:gaussian_tails}!
Indeed, we consider centered random variables $X_i-\mathbf{E}X_i$ and
from our discussion on the application of Rosenthal's inequality we see that $b = O(1/n)$ and $v^2=O(1/n)$. For the missing mass $\epsilon \leqslant 1$, we obtain the tail $\mathrm{e}^{-\Omega(n\epsilon^2)}$.

\subsubsection{Application: Missing Mass in Birthday Paradox}

We revisit Example 2 from~\cite{ben2017concentration}, which discuss the missing mass problem in the \emph{birthday paradox} setting. We have $m$ elements uniformly distributed so that $p_i = \frac{1}{m}$ for $i=1,\ldots,m$, and $n=\Theta(m^{1/2})$.
The authors of ~\cite{ben2017concentration} show that they were not able to obtain matching (up to constants in exponents) bounds on the lower and upper tail, because missmatching variance proxies $v_{+}^2$ and $v_{-}^2$.

We solve this problem and obtain matching lower and upper tails. 

\begin{corollary}[Confidence for Missing Mass in Birthday Paradox]\label{cor:birthday}
For the setting as above let $\sigma^2 = \sum_{i=1}^{m}\mathbf{Var}[Y_i]$. 
Then we have $\sigma^2 = \Theta(n^{-3})$ and the upper tail is
\begin{align}
\Pr[ M-\mathbf{E}M>\epsilon] \leqslant \exp\left( -\Omega (\epsilon^2/\sigma^2) \right),\quad \epsilon > 0.
\end{align}
Therefore the upper and lower tail are both $\mathrm{e}^{-\Omega(n^3\epsilon)}$. In other words the confidence interval are
\begin{align}
M \in [\mathbf{E}M -\epsilon,\mathbf{E}M+\epsilon] \quad \text{w.p.}\quad 1-\mathrm{e}^{ -\Omega (n^3\epsilon^2) }.
\end{align}
\end{corollary}
\begin{remark}[Variance proxy]
We can show that $\mathbf{Var}[M] =\Theta(\sigma^2)$, thus by the IID approximation we don't overestimate the variance by more than a constant factor!
\end{remark}


\begin{remark}[Super-linear exponent]
Note that the bound is better by a factor $\Omega(n^2)$ from the simple bound in~\Cref{eq:gaussian_tails}. This is a nice example showing that replacing $n$ by a distribution-dependent constant leads to huge improvements.
\end{remark}

\subsubsection{Numerical Evaluation}

To fairly compare with~\cite{ben2017concentration} we don't use their Theorem 3.9 but Proposition 3.8 which is bit stronger (and doesn't seem to suffer from issues with the range of $t$).
We find that our bounds give much better numerical estimates, likely due to constants in the exponents.

For starter we consider sampling from 4 elements with probabilities $p_1=0.1,p_2=0.2,p_3=0.3,p_4=0.4$, number of samples $n=10$
and $\epsilon=0.25$. The bound we obtain 
using \Cref{cor:main} is much better than when using poissonization, see~\Cref{fig:example}.
\begin{figure}[ht!]
\centering
\includegraphics[scale=0.5]{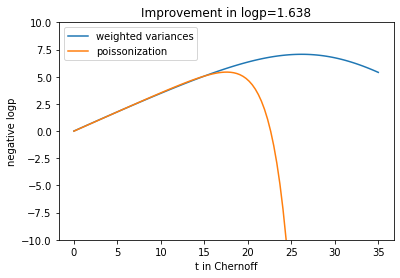}
\caption{Optimization of negative exponent in Chernoff's inequality using weighted variances (this paper) and poissonization (\cite{ben2017concentration}). Our method improves confidence bounds by a factor of $5$.}
\label{fig:example}
\end{figure}

Then we move to a wider evaluation, summarized in \Cref{tab:summary}. We sampled weights $p_i$ at random following Dirichlet distribution with parameter $\alpha=\frac{1}{2}$ (which gives uninformative prior).
\begin{table}[ht!]
\centering
\begin{tabular}{|c|c|c|c|c|c|}
\hline
$m$ & $n$ & av. missing mass & $\epsilon$ & $\mathrm{logp}$ improvement & std error \\
\hline
\multirow{2}{*}{30} & \multirow{2}{*}{30} & \multirow{2}{*}{ 0.07 } &0.01 & 0.30 & 0.02 \\
\cline{4-6} 
& &  & 0.1 & 4.72 & 1.44 \\
\hline
\end{tabular}
\caption{Summary of empirical evaluation.}
\label{tab:summary}
\end{table}
We observe significant improvement, particularly for larger values of $\epsilon$. The code is made available on \url{https://github.com/maciejskorski/missing_mass}.




\section{Preliminaries}

The moment generating function of a random variable is $\mathrm{MGF}_X(t) = \mathbf{E}\mathrm{e}^{tX}$.
The standard tool for deriving concentration inequalities is Chernoff's inequality, which states that $\Pr[X>\epsilon] < \mathrm{e}^{-t\epsilon}\cdot\mathbf{E}\mathrm{e}^{tX}$  for any real $t$; the best bound is obtained by optimizing the choice of $t$.
The best choice of $t$ is equivalent to maximizing $t\epsilon-\log \mathrm{MGF}_X(t)$ over $t$, e.g. finding the Legendre's dual of $\log\mathrm{MGF}_X$.

\section{Proofs}



\subsection{Proof of \Cref{main:lemma}}

By independence we have
\begin{align}
\mathbf{E}\exp\left(t\sum_{i=1}^{n} X_i\right)  = \prod_{i=1}^{n}\mathbf{E}\exp(t X_i)
\end{align}
We write the right side as $\prod_i (1+z_i)$ where $z_i = \mathbf{E}\exp(t X_i) - 1$ and apply the AM-GM inequality
\begin{align}
\mathbf{E}\exp\left(t\sum_{i=1}^{n} X_i\right)  = \prod_{i=1}^{n}(1+z_i) \leqslant \left(1+n^{-1}\sum_{i=1}^{n} z_i\right)
\end{align}
(we have $z_i\geqslant 0$ by Jensen's inequality and zero-mean property). The result follows by utilizing the Taylor's expansion $\mathbf{E}\exp(t X_i)-1 = \sum_{k\geqslant 1}\frac{t^k}{k!}\mathbf{E}|X_i|^k$ and the fact that $X_i$ are centered, so that terms with $k=1$ vanish.

\subsection{Proof of \Cref{cor:main}}\label{sec:proof_thm_mass}

Recall that for the missing mass problem it suffices to derive concentration bounds for
\begin{align}
M = \sum_i X_i
\end{align}
with \emph{independent summands} distributed as
\begin{align}
    X_i \overset{d}{=} p_i\cdot \mathrm{Bern}(q_i),\quad q_i = (1-p_i)^n,\quad p_i = \Pr[Y=i].
\end{align}
We shall apply \Cref{main:lemma} with $X_i$ replaced by $X_i-\mathbf{E}X_i$ (so that they are centered as required).

By definition
\begin{align}
\mathbf{E}|X_i-\mathbf{E}X_i|^k = p_i^k \cdot\left(q_i(1-q_i)^k + (1-q_i) q_i^k \right) = q_i p_i^k(1-q_i)^k + (1-q_i)p_i^k q_i^k
\end{align}
Define $\phi(x) = x^{-2}(\exp(x)-1-x) $. Summing over $k$ in the equation above we obtain for every $i$
\begin{align}
\sum_{k\geqslant 2}\frac{t^k}{k!}\mathbf{E}|X_i-\mathbf{E}X_i|^k = q_i (t p_i (1-q_i))^2\phi(t p_i (1-q_i)) + (1-q_i)(t p_i q_i)^2\phi(t p_i q_i)
\end{align}
Since $\mathbf{Var}[Y_i] = p_i^2 q_i(1-q_i)$ we have
\begin{align}
\sum_{k\geqslant 2}\frac{t^k}{k!}\mathbf{E}|X_i-\mathbf{E}X_i|^k =t^2\cdot \mathbf{Var}[Y_i] \cdot (q_i \phi(t p_i q_i) + (1-q_i)\phi(t p_i (1-q_i))
\end{align}
The lemma now gives
\begin{align}
\mathbf{E}[\mathrm{e}^{t(M-\mathbf{E}M)}] \leqslant (1+t^2 \sigma^2/ n)^n \leqslant \mathrm{e}^{t^2\sigma^2}
\end{align}
where $\sigma^2 =  \sum_i (q_i \phi(t p_i q_i) + (1-q_i)\phi(t p_i (1-q_i))\mathbf{Var}[Y_i]$. This finishes the proof of the moment generating function bound.
\subsection{Proof of \Cref{cor:rosenthal}}

We derive a bound on the moments. 
\begin{claim}
For $k\geqslant 2$ we have
\begin{align}
(\sum_i \mathbf{E}|X_i-\mathbf{E}X_i|^k)^{1/k} = k\cdot O(n^{-1+1/k}) 
\end{align}
\end{claim}
\begin{proof}[Proof of Claim]
Since $|X_i-\mathbf{E}X_i| \leqslant p_i$ we have $\mathbf{E}|X_i-\mathbf{E}X_i|^k \leqslant p_i^{k-2}\mathbf{Var}[X_i]$,
and thus 
\begin{align}
\mathbf{E}|X_i-\mathbf{E}X_i|^k \leqslant p_i^{k}\mathrm{e}^{-np_i}
\end{align}
Note that the function $p \to p^{k}\mathrm{e}^{-np}$ for $k,n>0$ is maximized for $p=k/n$, therefore we have
\begin{align}
\sum_i \mathbf{E}|X_i-\mathbf{E}X_i|^k \leqslant \sum_i p_i \cdot (k/n)^{k-1}\mathrm{e}^{-k}  \leqslant (k/\mathrm{e})^{k-1}\cdot n^{1-k}
\end{align} 
which gives the claimed bound.
\end{proof}
Now \Cref{lemma:Rosenthal} implies
\begin{align}
\left(\mathbf{E}\left|\sum_i (X_i-\mathbf{E}X_i)\right|^k\right)^{1/k} = O(k^{1/2} n^{-1/2}) +O(k^2 n^{-1+1/k}). 
\end{align}
Let's compare the two terms on the right-hand side. We have
$ k^2 n^{-1+1/k} = O(k^{1/2} n^{-1/2})$ if and only if $n^{1/k} k^{3/2} = O( n^{1/2})$. 
This is true as long $k=O(n^{1/3})$; this is because the derivative test shows that the left side is 
decreases when $k<2/3\cdot \log n$ and increases when $k>2/3\cdot \log n$ thus the maximum occurs at $k=2$ or $k=\Theta(n^{1/3})$.
Thus
\begin{align}
\left(\mathbf{E}\left|\sum_i (X_i-\mathbf{E}X_i)\right|^k\right)^{1/k} = O(k^{1/2} n^{-1/2}) ,\quad k = O(n^{1/3}).
\end{align}
By applying Markov's inequality we obtain the tail bound of $\exp(-\Omega(n\epsilon^2))$ provided that 
$n\epsilon^2 = O(n^{1/3})$ or $\epsilon = O(n^{-1/3})$. 
Here we formally need something stronger than stochastic domination of MGFs, namely the domination of moments which also follows by negative dependence.

\subsection{Proof of \Cref{eq:remark2}}

Note that $\mathbf{Var}[Y_i] \leqslant \mathbf{E}(Y_i)^2 = p_i^2q_i\leqslant p_i^2\mathrm{e}^{-np_i}$,
$\theta_i \leqslant \phi(t p_i)$ and $\phi(x) = \sum_{r\geqslant 2} \frac{x^{r-2}}{r!}$,
therefore
\begin{align}
\sigma^2=2\sum_i \theta_i \mathbf{Var}[Y_i] \leqslant 2 \sum_i p_i^2 \mathrm{e}^{-np_i} \sum_{r\geqslant 2} \frac{(t p_i )^{r-2}}{r!} = 2t^{-2}\sum_{r\geqslant 2}\left(\frac{t}{n}\right)^{r} \mathbf{E}K_r(n).
\end{align}

\subsection{Proof of \Cref{cor:birthday}}

We use the notation as in \Cref{sec:proof_thm_mass}, that is
we consider independent $X_i = p_i\mathrm{Bern}(q_i)$, where $p_i = 1/m$ and $q_i = (1-p_i)^n$.

\begin{claim}
We have $\mathbf{E}M= 1-\Theta(n^{-1})$.
\end{claim}
\begin{proof}[Proof of Claim]
We have $\mathbf{E}M= \sum_i\mathbf{E}X_i$ and $\mathbf{E}X_i = p_i(1-p_i)^{n}$. Since all $p_i$ are equal we get
$\mathbf{E}M = (1-1/m)^{n}$. Since $1/m\cdot n = \Theta(1/n)$ we have $\mathbf{E} M  = 1-\Theta(n^{-1})$.
\end{proof}
Since $M\leqslant 1$ it follows that we can restrict to $\epsilon = 1-M = O(1/n)$ when discussing tails (for bigger values of $\epsilon$ the true tail probability is zero so the bounds hold trivially).

\begin{claim}
We have $\mathbf{Var}[X_i]  = \Theta(n^{-5})$, and thus $\sum_i\mathbf{Var}[X_i] = \Theta(n^{-3})$.
\end{claim}
\begin{proof}[Proof of Claim]
By definition $\mathbf{Var}[X_i] = p_i^2 q_i(1-q_i)$. Since $q_i = (1-p_i)^n$ and since $p_i\cdot n = \Theta(1/n)$ we have
$q_i = 1-\Theta(1/ n)$, therefore $\mathbf{Var}[X_i] = \Theta(m^{-2}n^{-1})$.
\end{proof}

Consider now the weights $\theta_i$, we have
\begin{claim}
For any $t = O(n^2)$ we have $\theta_i = \Theta(1)$.
\end{claim}
\begin{proof}[Proof of Claim]
Since $p_i = 1/m = O(1/n^2)$ and $t=O(n^2)$ we have $t p_i = O(1)$. Since $\phi$ is bounded around zero, and $\phi(0) = 1/2$ we have 
$\phi(x) = \Theta(1)$ for $x=O(1)$. Since $t p_i (1-q_i) \leqslant t p_i = O(1)$ and $t p_i q_i \leqslant t p_i = O(1)$ we conclude that
$\phi(t p_i q_i) = \Theta(1)$ and $\phi(t p_i(1-q_i)) = \Theta(1)$. Now the weight $\theta_i$ is a convex combination of those two, and thus $\theta_i = \Theta(1)$ as well.
\end{proof}

Now we have $\sum_i \theta_i \mathbf{Var}[X_i] = \Theta(n^{-3})$ when $0< t \leqslant n^2$, so
by \Cref{cor:main} we obtain
\begin{claim}
For the setup as above and $\sigma^2 = \Theta(n^{-3})$ we have
\begin{align}
\mathbf{E}\exp(t(M-\mathbf{E}M)) \leqslant \exp(O(t^2\sigma^2)),\quad 0<t\leqslant n^2.
\end{align}
\end{claim}

This is known to imply, by Chernoff's inequality~\cite{wainwright2019high}, the Bernstein-type bound
\begin{align}
\Pr[M-\mathbf{E}M \geqslant \epsilon] \leqslant \exp\left(-\frac{c\cdot\epsilon^2}{\sigma^2 + \epsilon/n^2}\right).
\end{align}
for a universal constant $c$.  Since we proved that we can assume $\epsilon = O(1/n)$, we get
$\epsilon/n^2 = O(\sigma^2/n)$ because $\sigma^2 = \Theta(n^{-3})$. Therefore the tail bound simplifies to $\mathrm{e}^{-\Omega(\epsilon^2/\sigma^2)}$.

\subsection{Proof of \Cref{cor:flexible}}

We note that $\mathbf{E}X_i^{k} \leqslant \mathbf{E}(X_i^{+})^k$ for any positive integer $k$, therefore in \Cref{main:lemma}
we get the bound 
\begin{align}
\mathbf{E}\exp\left(t S\right) \leqslant \left(1+ \frac{v^2 t^2}{m} \sum_{k\geqslant 2}\frac{ t^{k-2} b^{k-2}}{2} \right)^m = 
\left(1+ \frac{v^2 t^2}{m(1-tb)} \right)^m
\end{align}
where in the second step we assume $|tb|<1$; the bound on the moment generating function follows by the inequality $(1+x/m)^m\leqslant \exp(x)$ for positive $x$ and $m$.
The tail bound follows by  Chernoff's inequality and some optimization of the parameter
(see for example~\cite{wainwright2019high}).

\section{Numerical Comparison}

\bibliographystyle{plain}
\bibliography{citations}

\begin{thebibliography}{10}

\bibitem{acharya2018improved}
Jayadev Acharya, Yelun Bao, Yuheng Kang, and Ziteng Sun.
\newblock Improved bounds for minimax risk of estimating missing mass.
\newblock In {\em 2018 IEEE International Symposium on Information Theory
  (ISIT)}, pages 326--330. IEEE, 2018.

\bibitem{aldaz2008selfimprovemvent}
JM~Aldaz.
\newblock Selfimprovemvent of the inequality between arithmetic and geometric
  means.
\newblock {\em arXiv preprint arXiv:0807.1788}, 2008.

\bibitem{ben2017concentration}
Anna Ben-Hamou, St{\'e}phane Boucheron, Mesrob~I Ohannessian, et~al.
\newblock Concentration inequalities in the infinite urn scheme for occupancy
  counts and the missing mass, with applications.
\newblock {\em Bernoulli}, 23(1):249--287, 2017.

\bibitem{berend2013concentration}
Daniel Berend, Aryeh Kontorovich, et~al.
\newblock On the concentration of the missing mass.
\newblock {\em Electronic Communications in Probability}, 18, 2013.

\bibitem{boucheron2013concentration}
St{\'e}phane Boucheron, G{\'a}bor Lugosi, and Pascal Massart.
\newblock {\em Concentration inequalities: A nonasymptotic theory of
  independence}.
\newblock Oxford university press, 2013.

\bibitem{buldygin2013sub}
V~Buldygin and K~Moskvichova.
\newblock The sub-gaussian norm of a binary random variable.
\newblock {\em Theory of probability and mathematical statistics}, 86:33--49,
  2013.

\bibitem{DBLP:conf/isit/ChandraT19}
Prafulla Chandra and Andrew Thangaraj.
\newblock Concentration and tail bounds for missing mass.
\newblock In {\em {IEEE} International Symposium on Information Theory, {ISIT}
  2019, Paris, France, July 7-12, 2019}, pages 1862--1866. {IEEE}, 2019.

\bibitem{dubhashi1996balls}
Devdatt~P Dubhashi and Desh Ranjan.
\newblock Balls and bins: A study in negative dependence.
\newblock {\em BRICS Report Series}, 3(25), 1996.

\bibitem{ibragimov1998exact}
Rustam Ibragimov and Sh~Sharakhmetov.
\newblock On an exact constant for the rosenthal inequality.
\newblock {\em Theory of Probability \& Its Applications}, 42(2):294--302,
  1998.

\bibitem{joag1983negative}
Kumar Joag-Dev and Frank Proschan.
\newblock Negative association of random variables with applications.
\newblock {\em The Annals of Statistics}, pages 286--295, 1983.

\bibitem{kearns2013large}
Michael Kearns and Lawrence Saul.
\newblock Large deviation methods for approximate probabilistic inference.
\newblock {\em arXiv preprint arXiv:1301.7392}, 2013.

\bibitem{khanloo2015novel}
Bahman Yari~Saeed Khanloo and Gholamreza Haffari.
\newblock Novel bernstein-like concentration inequalities for the missing mass.
\newblock In Marina Meila and Tom Heskes, editors, {\em Proceedings of the
  Thirty-First Conference on Uncertainty in Artificial Intelligence, {UAI}
  2015, July 12-16, 2015, Amsterdam, The Netherlands}, pages 425--434. {AUAI}
  Press, 2015.

\bibitem{mcallester2003concentration}
David McAllester and Luis Ortiz.
\newblock Concentration inequalities for the missing mass and for histogram
  rule error.
\newblock {\em Journal of Machine Learning Research}, 4(Oct):895--911, 2003.

\bibitem{mcallester2000convergence}
David~A McAllester and Robert~E Schapire.
\newblock On the convergence rate of good-turing estimators.
\newblock In {\em COLT}, pages 1--6, 2000.

\bibitem{ortiz2003concentration}
Luis~E Ortiz and David~A McAllester.
\newblock Concentration inequalities for the missing mass and for histogram
  rule error.
\newblock In {\em Advances in Neural Information Processing Systems}, pages
  367--374, 2003.

\bibitem{pinelis1994optimum}
Iosif Pinelis.
\newblock Optimum bounds for the distributions of martingales in banach spaces.
\newblock {\em The Annals of Probability}, pages 1679--1706, 1994.

\bibitem{wainwright2019high}
Martin~J Wainwright.
\newblock {\em High-dimensional statistics: A non-asymptotic viewpoint},
  volume~48.
\newblock Cambridge University Press, 2019.

\end{thebibliography}

\end{document}